\documentclass[11pt, a4paper]{article}

\usepackage{color,url,latexsym,amssymb,amsthm,amsmath,amsxtra,amsfonts}
\usepackage{graphicx}
\usepackage{amsfonts}
\usepackage{amsbsy}
\usepackage{amssymb}
\usepackage{amsmath}
\usepackage{amsthm}
\theoremstyle{plain}
\newtheorem{theorem}{Theorem}[section]
\newtheorem{lemma}{Lemma}[section]
\newtheorem{corollary}{Corollary}[section]
\newtheorem{proposition}{Proposition}[section]
\theoremstyle{definition}
\newtheorem{definition}{Definition}[section]
\newtheorem{remark}{Remark}[section]

\begin{document}

\title{A vertical Liouville subfoliation on the cotangent bundle of a Cartan space and some related structures}
\author{Cristian Ida and Adelina Manea}
\date{}
\maketitle

\begin{abstract}
In this paper we study some problems related to a vertical Liouville distribution (called vertical Liouville-Hamilton distribution) on the cotangent bundle of a Cartan space. We study the existence of some linear connections of Vr\u{a}nceanu type on Cartan spaces related to some foliated structures. Also, we identify a certain $(n,2n-1)$--codimensional subfoliation $(\mathcal{F}_V,\mathcal{F}_{C^*})$ on $T^*M_0$ given by vertical foliation $\mathcal{F}_V$ and the line foliation $\mathcal{F}_{C^*}$ spanned by the vertical Liouville-Hamilton vector field $C^*$ and we give a triplet of basic connections adapted to this subfoliation. Finally, using the vertical Liouville foliation $\mathcal{F}_{V_{C^*}}$ and the natural almost complex structure on $T^*M_0$ we study some aspects concerning the cohomology of $c$--indicatrix cotangent bundle. 
\end{abstract}

\medskip

\begin{flushleft}
\strut \textbf{2010 Mathematics Subject Classification}: 53B40, 53C12, 53C60.

\end{flushleft}

\section{Introduction and preliminaries}
\setcounter{equation}{0}

\subsection{Introduction}
The study of interrelations between the geometry of some natural foliations on the tangent manifold of a Finsler space and the geometry of the Finsler space itself was initiated and intensively studied by Bejancu and Farran \cite{Bej}. The main idea of their paper is to emphasize the importance of some foliations which exist on the tangent bundle of a Finsler space $(M, F)$, in studying the
differential geometry of $(M, F)$ itself. In this direction, in the last decades, the geometrical aspects determined by these foliations on the tangent manifold of a Finsler space were studied \cite{C-C, Ma-I, P-T-Z, P-N}. On the other hand in a very recent paper \cite{A-R1} a similar study of some natural foliations in cotangent bundle $T^*M$ of a Cartan space $(M, K)$ is given. It is shown that geometry of these foliations is closely related to the geometry of the Cartan space $(M, K)$ itself. This approach is used to obtain new characterizations of Cartan spaces with negative constant curvature.

The aim of our paper is to continue the study of these foliations on the cotangent bundle of a Cartan space $(M,K)$ from more points of view as: the study of some linear connections of Vr\u{a}nceanu type on Cartan spaces related to foliated structures, the study of a certain $(n,2n-1)$--codimensional subfoliation $(\mathcal{F}_V,\mathcal{F}_{C^*})$ on $T^*M_0$ given by vertical foliation $\mathcal{F}_V$ and the line foliation $\mathcal{F}_{C^*}$ spanned by the vertical Liouville-Hamilton vector field $C^*$ and the existence of a triplet basic connections adapted to this subfoliation. Also, some aspects concerning the cohomology of $c$--indicatrix cotangent bundle are investigated. The notions are introduced by analogy with corresponding notions on the tangent manifold of a Finsler space, \cite{I-P, Ma-I}. 

The paper is organized as  follows: In the preliminary subsection we briefly recall some basic facts on the geometry of a Cartan space $(M,K)$ and we present the  almost K\"{a}hlerian model $(T^*M_0,G,J)$ of the cotangent manifold $T^*M_0=T^*M-\{{\rm zero \,\,section}\}$ together with the Riemannian metric $G$ given by Sasaki type lift of the fundamental metric tensor $g^{ij}=\frac{1}{2}\frac{\partial^2K^2}{\partial p_i\partial p_j}$ and with the natural almost complex structure $J$. In the second section, following an argument inspired from \cite{B-F1}, we define a vertical Liouville distribution $V_{C^*}$ on $T^*M_0$ as the complementary orthogonal distribution in the vertical distribution  $V(T^*M_0)$ to the line distribution spanned by the vertical Liouville-Hamilton vector field $C^*=p_i\frac{\partial}{\partial p_i}$ and we prove that the distribution $V_{C^*}$ is an integrable one. Also, we find geometric properties of both leaves of  vertical Liouville distribution $V_{C^*}$ and the vertical distribution $V(T^*M_0)$. We notice that the vertical Liouville distribution $V_{C^*}$ will be an important tool in the study of future problems in this paper.  In the third section, following some ideas from \cite{B-F}, we present some linear connections of Vr\u{a}nceanu type on a Cartan space, related with the vertical foliation $\mathcal{F}_V$ and vertical Hamilton-Liouville foliation $\mathcal{F}_{V_{C^*}}$ on $T^*M_0$. In the fourth section, following \cite{C-M}, we briefly recall the notion of a $(q_1,q_2)$-codimensional subfoliation on a manifold and we identify a $(n,2n-1)$-codimensional subfoliation $(\mathcal{F}_V,\mathcal{F}_{C^*})$ on the cotangent manifold $T^*M_0$ of a Cartan space $(M,K)$, where $\mathcal{F}_V$ is the vertical foliation and $\mathcal{F}_{C^*}$ is the line foliation spanned by the vertical Liouville-Hamilton vector field $C^*$. Firstly, we make a general approach about basic connections on the normal bundles related to this subfoliation and next a triple of adapted basic connections with respect to this subfoliation is given. A similar study on the tangent manifold of a Finsler space $(M,F)$ is given in \cite{Ma-I}. In the last section, using the vertical Liouville-Hamilton vector field $C^*$ and the natural almost complex structure $J$ on $T^*M_0$, we give an adapted basis in $T(T^*M_0)$. Next we prove that the $c$--indicatrix cotangent bundle $I(M,K)(c)$ of $(M,K)$ is a $CR$--submanifold of the almost K\"{a}hlerian manifold $(T^*M_0,G,J)$ and we study some cohomological properties of $I(M,K)(c)$ in relation with classical cohomology of $CR$-submanifolds, \cite{C1}.

\subsection{Preliminaries and notations}
In this subsection we briefly recall some basic facts from the geometry of a Cartan space $(M,K)$. For more see \cite{An, M1, M2, M3, M-H-S}.

Let $M$ be an $n$-dimensional $C^{\infty}$ manifold and $\pi^*:T^*M\rightarrow M$ its cotangent bundle. If $(x^i)$, $i=1,\ldots,n$ are local coordinates in a local chart $U$ on $M$, then $(x^i,p_i)$, $i=1,\ldots,n$ will be taken as local coordinates in the local chart $(\pi^*)^{-1}(U)$ on $T^*M$ with the momenta $(p_i)$ provided by $p=p_idx^i$ where $p\in T^*_xM$, $x\in M$ and $\{dx^i\}$ is the natural basis of $T^*_xM$. The indices $i,j,k,\ldots$ will run from $1$ to $n$ and the Einstein convention on summation will be used. 

Let us consider $\left\{\frac{\partial}{\partial x^i},\frac{\partial}{\partial p_i}\right\}$, $i=1,\ldots,n$ the natural basis in $T_{(x,p)}(T^*M)$ and $\{dx^i,dp_i\}$ the dual basis of it.  The kernel $V_{(x,p)}(T^*M)$ of the differential $d\pi^*:T_{(x,p)}(T^*M)\rightarrow T_xM$ is called the \textit{vertical} subspace of $T_{(x,p)}(T^*M)$ and the mapping $(x,p)\mapsto V_{(x,p)}(T^*M)$ is a regular distribution on $T^*M$ called the \textit{vertical distribution}. This is integrable and defines the vertical foliation $\mathcal{F}_V$ with the leaves characterized by $x^k=constant$ and it is locally spanned by $\left\{\frac{\partial}{\partial p_i}\right\}$. The vector field $C^*=p_i\frac{\partial}{\partial p_i}$ is called the \textit{vertical Liouville-Hamilton} vector field and $\omega=p_idx^i$ is called the Liouville $1$--form on $T^*M$. Then $\Omega=d\omega=dp_i\wedge dx^i$ is the canonical symplectic structure on $T^*M$. For an easier handling of the geometrical objects
on $T^*M$ it is usual to consider a supplementary distribution to the vertical distribution, $(x,p)\mapsto H_{(x,p)}(T^*M)$, called the \textit{horizontal} distribution and to report all geometrical objects on $T^*M$ to the decomposition
\begin{equation}
T_{(x,p)}(T^*M)=H_{(x,p)}(T^*M)\oplus V_{(x,p)}(T^*M).
\label{I4}
\end{equation}
The horizontal distribution is taken as being locally spanned by the local vector fields
\begin{equation}
\frac{\delta}{\delta x^i}:=\frac{\partial}{\partial x^i}+N_{ij}(x,p)\frac{\partial}{\partial p_j}.
\label{I5}
\end{equation}
The horizontal distribution is called also a \textit{nonlinear connection} on $T^*M$ and the functions $N_{ij}$ are called the local coefficients of this nonlinear connection. It is important to note that
any regular Hamiltonian on $T^*M$ determines a nonlinear connection whose local coefficients
verify $N_{ij}=N_{ji}$. The basis $\left\{\frac{\delta}{\delta x^i},\frac{\partial}{\partial p_i}\right\}$ is adapted to the decomposition \eqref{I4}. The dual of it is $\{dx^i,\delta p_i:=dp_i-N_{ji}dx^j\}$.

According to \cite{M-H-S} a \textit{Cartan structure} on $M$ is a function $K:T^*M\rightarrow[0,\infty)$ which has the following properties:
\begin{enumerate}
\item[i)] $K$ is $C^{\infty}$ on $T^*M_0:=T^*M-\{{\rm zero\,\,section}\}$;
\item[ii)] $K(x,\lambda p)=\lambda K(x,p)$ for all $\lambda>0$;
\item[iii)] the $n\times n$ matrix $(g^{ij})$, where $g^{ij}=\frac{1}{2}\frac{\partial^2 K^2}{\partial p_i\partial p_j}$, is positive definite at all points of $T^*M_0$. 
\end{enumerate}
We notice that in fact $K(x,p)>0$, whenever $p\neq0$.
\begin{definition}
The pair $(M,K)$ is called a Cartan space.
\end{definition}
Let us put
\begin{equation}
p^i=\frac{1}{2}\frac{\partial K^2}{\partial p_i}\,,\,C^{ijk}=-\frac{1}{4}\frac{\partial^3K^2}{\partial p_i\partial p_j\partial p_k}.
\label{I6}
\end{equation}
The properties of $K$ imply that
\begin{equation}
p^i=g^{ij}p_j\,,\,p_i=g_{ij}p^j\,,\,K^2=g^{ij}p_ip_j=p_ip^i\,,\,C^{ijk}p_k=C^{ikj}p_k=C^{kij}p_k=0,
\label{I7}
\end{equation}
where $(g_{ij})$ is the inverse matrix of $(g^{ji})$.

One considers the formal Christoffel symbols by
\begin{equation}
\gamma^i_{jk}(x,p):=\frac{1}{2}g^{is}\left(\frac{\partial g_{js}}{\partial x^k}+\frac{\partial g_{sk}}{\partial x^j}-\frac{\partial g_{jk}}{\partial x^s}\right),
\label{I8}
\end{equation}
and the contractions $\gamma^0_{jk}(x,p):=\gamma^i_{jk}(x,p)p_i$, $\gamma^0_{j0}:=\gamma^i_{jk}p_ip^k$. Then the functions
\begin{equation}
N_{ij}(x,p)=\gamma^0_{ij}(x,p)-\frac{1}{2}\gamma^0_{h0}(x,p)\frac{\partial g_{ij}}{\partial p_h}(x,p),
\label{I9}
\end{equation}
define a nonlinear connection on $T^*M$. This nonlinear connection was discovered by  Miron \cite{M3} and is called the \textit{canonical nonlinear connection} of $(M,K)$. We also  notice that the coefficients  from \eqref{I9} satisfies $N_{ij}(x,p)=N_{ji}(x,p)$ and are positively homogeneous of degree $1$ in momenta. For details we refer to Ch. 6 in \cite{M-H-S}.  Thus a decomposition \eqref{I4} holds. From now on we shall use only the canonical nonlinear connection given by \eqref{I9}.

To a Cartan space $(M,K)$ we can associate some important geometrical object fields on the manifold $T^*M_0$. Namely, the $N$-lift $G$ of the fundamental tensor $g^{ij}$, the almost complex structure $J$, etc. If $N$ is the canonical nonlinear of $(M,K)$, thus $(G,J)$ determine an almost Hermitian structure, which is derived only from the fundamental function $K$ of the Cartan space.

The $N$-lift of the fundamental tensor field $g^{ij}$ of the space $(M,K)$ is defined by
\begin{equation}
G=g_{ij}dx^i\otimes dx^j+g^{ij}\delta p_i\otimes\delta p_j
\label{I10}
\end{equation}
and there is a natural almost complex structure $J$ on $T^*M_0$ which is locally given by
\begin{equation}
J=-g_{ij}\frac{\partial}{\partial p_i}\otimes dx^j+g^{ij}\frac{\delta}{\delta x^i}\otimes \delta p_j\,,\,J\left(\frac{\delta}{\delta x^i}\right)=-g_{ij}\frac{\partial}{\partial p_j}\,,\,J\left(\frac{\partial}{\partial p_i}\right)=g^{ij}\frac{\delta}{\delta x^j}.
\label{I11}
\end{equation}
Thus, according to \cite{M-H-S}, $(T^*M_0,G,J)$ has a model of almost K\"{a}hlerian manifold with the fundamental form $\Omega$ given by $\Omega(X,Y)=G(JX,Y)$ and locally expressed by
\begin{equation}
\Omega=\delta p_i\wedge dx^i=dp_i\wedge dx^i.
\label{I12}
\end{equation}

\section{A vertical Liouville distribution on $T^*M_0$}
\setcounter{equation}{0}
Following an argument inspired from \cite{B-F1} we define a vertical Liouville distribution on $T^*M_0$ as the complementary orthogonal distribution in $V(T^*M_0)$ to the line distribution spanned by the vertical Liouville-Hamilton vector field $C^*$ and we prove that this distribution is an integrable one.

By \eqref{I7} we have
\begin{equation}
G(C^*,C^*)=K^2.
\label{II1}
\end{equation}
By means of $G$ and $C^*$, we define the vertical one form $\zeta$ by
\begin{equation}
\zeta(X)=\frac{1}{K}G(X,C^*)\,\,\forall\,X\in\Gamma(V(T^*M_0)).
\label{II2}
\end{equation}
Denote by $\left\{C^*\right\}$ the line vector bundle over $T^*M_0$ spanned by $C^*$ and define the \textit{Liouville distribution} as the complementary orthogonal distribution $V_{C^*}$ to $\left\{C^*\right\}$ in $V(T^*M_0)$ with
respect to $G$. Hence, $V_{C^*}$ is defined by $\zeta$, that is, we have
\begin{equation}
\Gamma\left(V_{C^*}\right)=\{X\in\Gamma(V(T^*M_0))\,:\,\zeta(X)=0\}.
\label{II3}
\end{equation}
Thus, any vertical vector field $X=X_i\frac{\partial}{\partial p_i}$ can be expressed as follows:
\begin{equation}
X=PX+\frac{1}{K}\zeta(X)C^*,
\label{II4}
\end{equation}
where $P$ is the projection morphism of $V(T^*M_0)$ on $V_{C^*}$. By direct calculations, we obtain
\begin{equation}
G(X,PY)=G(PX,PY)=G(X,Y)-\zeta(X)\zeta(Y),\,\,\forall\,X,Y\in\Gamma(V(T^*M_0)).
\label{II5}
\end{equation}
Then the local components of $\zeta$ and $P$ with respect to the basis $\{\delta p_i\}$ and $\left\{\delta p_i\otimes\frac{\partial}{\partial p_i}\right\}$, respectively, are given by
\begin{equation}
\zeta^i=\frac{p^i}{K}\,,\,P_j^i=\delta^i_j-\frac{\zeta^ip_j}{K},
\label{II6}
\end{equation}
where $\delta^i_j$ are the components of the Kronecker delta.
\begin{remark}
\label{rr0}
Let us consider the horizontal projector $h =\frac{\delta}{\delta x^i}\otimes dx^i$
and the vertical
projector $v =\frac{\partial }{\partial p_i}\otimes \delta p_i $, with respect to the canonical nonlinear connection $N_{ij}$ . Then
by Proposition 6.4.3, pag. 147 \cite{M-H-S}, we have $\frac{\delta K}{\delta x^i}=0$ which implies $d_hK = 0$. Then the
vertical 1-form $\zeta$ can be given as follows: $\zeta = d_vK = dK$.
\end{remark}
\begin{remark}
\label{rr1}
The projector $P$ can be related in terms of the \textit{angular metric tensor} of the Cartan space $(M,K)$, which is defined as in Finsler geometry as follows
\begin{equation}
\label{r1}
h^{ij}=g^{ij}-\frac{1}{K^2}p^ip^j.
\end{equation}
More exactly,  by \eqref{II6} and \eqref{r1} it is easy to see that $P^i_j=g_{jk}h^{ki}$. Also, in a similar manner
as in Finsler geometry \cite{C-M-S} it follows that the regularity condition for $g_{ij}$ is equivalent
with the fact that $h_{ij}$ is positive quasi-definite.
\end{remark}
Now, by \eqref{II6} the rank of the projector $P$ is $n-1$ and taking into account that $P^i_jp_i=0$ it follows that $V_{C^*}$ is an $(n-1)$--dimensional vertical sub-distribution, orthogonal to $C^*$ in $V(T^*M_0)$, locally spanned by the vertical vector fields $\left\{\frac{\overline{\partial}}{\overline{\partial}p_i}\right\}$, $i=1,\ldots,n$, where
\begin{equation}
\label{r2}
\frac{\overline{\partial}}{\overline{\partial}p_i}=\frac{\partial}{\partial p_i}-\frac{p^i}{K^2}C^*=P^i_j\frac{\partial}{\partial p_j}.
\end{equation}
Taking  into account that $C^*(K)=K$ and $K\frac{\partial K}{\partial p_i}=p^i$ (see Proposition 6.4.3, pag. 147
\cite{M-H-S}), we obtain that an important property of the vertical Liouville sub-distribution $V_{C^*}$ is the following: 

\textit{For any} $Y\in\Gamma(V_{C^*})$ \textit{we have} $Y(K)=0$.
\begin{theorem}
\label{integrabil}
The vertical Liouville-Hamilton distribution $V_{C^*}$ is integrable and hence it defines a foliation on $T^*M_0$ denoted by $\mathcal{F}_{V_{C^*}}$.
\end{theorem}
\begin{proof}
The proof it follows using an argument similar to  Theorem 3.1 from \cite{B-F1}.
\end{proof}
\begin{remark}
The proof of Theorem \ref{integrabil} can be also obtained using an argument similar to \cite{B-Mu}. More exactly, if $\frac{\overline{\partial}}{\overline{\partial}p_i}$, $\frac{\overline{\partial}}{\overline{\partial}p_j}\in\Gamma(V_{C^*})\subset\Gamma(V(T^*M_0))$, then
\begin{equation}
\label{r3}
\left[\frac{\overline{\partial}}{\overline{\partial}p_i},\frac{\overline{\partial}}{\overline{\partial}p_j}\right]=A^{ij}_k\frac{\overline{\partial}}{\overline{\partial}p_k}+B^{ij}C^*,
\end{equation} 
for some locally defined functions $A^{ij}_k$ and $B^{ij}$, since $V(T^*M_0)=V_{C^*}\oplus\{C^*\}$ is integrable. Now, if we apply the vector fields in both sides of formula \eqref{r3} to the Cartan function $K$ and using the fact that $C^*(K)=K$ and $\frac{\overline{\partial}K}{\overline{\partial}p_i}=0$, we obtain $B^{ij}K=0$. This implies that $B^{ij}=0$, and then the formula \eqref{r3} says that the vertical Liouville distribution $V_{C^*}$ is integrable.
\end{remark}
By direct caculations, we obtain the following relations for the Lie brakets of vertical vector fileds adapted to the decomposition $V(T^*M_0)=V_{C^*}\oplus\{C^*\}$,
\begin{equation}
\left[\frac{\overline{\partial}}{\overline{\partial} p_i},\frac{\overline{\partial}}{\overline{\partial} p_j}\right]=\frac{1}{K^2}\left(p^i\frac{\overline{\partial}}{\overline{\partial} p_j}-p^j\frac{\overline{\partial}}{\overline{\partial} p_i}\right),\,\,
\left[\frac{\overline{\partial}}{\overline{\partial} p_i},C^*\right]=\frac{\overline{\partial}}{\overline{\partial} p_i},
\label{III7}
\end{equation}
for all $i,j=1,\ldots,n$, and the first relation of \eqref{III7} says also that $V_{C^*}$ is integrable.

Based on the above results, similarly to \cite{B-F1}, we may say that the geometry of the leaves of $\mathcal{F}_V$ should be derived from the geometry of the leaves of $\mathcal{F}_{V_{C^*}}$ and of integral curves
of $C^*$. In order to get this interplay, we consider a leaf $F_V$ of $\mathcal{F}_V$ given locally by $x^i=a^i$, $i=1,\ldots,n$, where the $a^i$'s are constants. Then, $g^{ij}(a,p)$ are
the components of a Riemannian metric $G_{F_V}=G|_{F_V}$ on $F_V$. Denote by  $\nabla$ the
Levi-Civita connection on $F_V$ with respect to $G_{F_V}$ and consider the Christoffel symbols
$C_i^{jk}$ of $\nabla$. Then we obtain the usual formula for $C_i^{jk}$, namely
\begin{equation}
C_i^{jk}(a,p)=-\frac{1}{2}g_{is}(a,p)\frac{\partial g^{sk}}{\partial p_j}(a,p)=g_{is}(a,p)C^{sjk}(a,p),
\label{II10}
\end{equation}
where $g_{is}(a,p)$ are the entries of the inverse matrix of the $n\times n$ matrix $\left(g^{si}(a,p)\right)$.
Contracting \eqref{II10} by $p_j$ , we deduce that
\begin{equation}
C_i^{jk}(a,p)p_j=0.
\label{II11}
\end{equation}
By straightforward calculations using \eqref{II11}, \eqref{II5} and \eqref{II6}, we obtain the covariant derivatives of $C^*$, $\zeta$ and $P$ in the following lemma:
\begin{lemma}
Let $(M,K)$ be a Cartan space. Then, on any leaf $F_V$ of $\mathcal{F}_V$, we have
\begin{equation}
\nabla_X\left(\frac{1}{K}C^*\right)=\frac{1}{K}PX,
\label{II12}
\end{equation}
\begin{equation}
\left(\nabla_X\zeta\right)Y=\frac{1}{K}G(PX,PY),
\label{II13}
\end{equation}
and
\begin{equation}
\left(\nabla_XP\right)Y=-\frac{1}{K^2}\left[G(PX,PY)C^*+K\zeta(Y)PX\right]
\label{II14}
\end{equation}
for any $X,Y\in\Gamma\left(TF_V\right)$.
\end{lemma}
Now, the following results can be easy obtained in a similar manner with the case of Finsler geometry, see \cite{B-F1}:
\begin{theorem}
Let $(M,K)$ be an $n$--dimensional Cartan space and $F_V$, $F_{V_{C^*}}$ and $\gamma$ be a leaf of $\mathcal{F}_V$, a leaf of $\mathcal{F}_{V_{C^*}}$ that lies in $F_V$, and an integral curve of $\frac{1}{K}C^*$, respectively. Then we have the following assertions:
\begin{enumerate}
\item[i)] $\gamma$ is a geodesic of $F_V$ with respect to $\nabla$.
\item[ii)] $F_{V_{C^*}}$ is totally umbilical immersed in $F_V$.
\item[iii)] $F_{V_{C^*}}$ lies in the indicatrix $I_a(M,K)=\{p\in T_a^*M_0\,:\,K(a,p)=1\}$ of $(M,K)$ and has constant mean curvature equal to $-1$.
\end{enumerate}
\end{theorem}
\begin{theorem}
Let $(M,K)$ be an $n$--dimensional Cartan space and $F_V$ be a leaf of the vertical foliation $\mathcal{F}_V$. Then the sectional curvature of any nondegenerate plane section on $F_V$ containing the vertical Liouville-Hamilton vector field is equal to zero.
\end{theorem}
\begin{corollary}
Let $(M,K)$ be an $n$--dimensional Cartan space. Then there exist no leaves of $\mathcal{F}_V$ which are positively or negatively curved.
\end{corollary}

\section{Vr\u{a}nceanu connections on the cotangent bundle of a Cartan space}
\setcounter{equation}{0}

 In this section, following some ideas from \cite{B-F}, we investigate the Vr\u{a}nceanu connections on a Cartan space, related with the vertical and vertical Liouville-Hamilton foliations. 
 
 First of all we remark that on a Cartan space $(M,K)$ there exists the canonical metrical $N$-linear connection, \cite{M-H-S}: $C\Gamma(N)=(H_{jk}^i,C^{jk}_i)$. Its local coefficients, with respect to adapted basis $\{\frac{\delta}{\delta x^i},\frac{\partial}{\partial p_i}\}$, are
 \begin{equation}
 H_{jk}^i=\frac{1}{2}g^{is}\left(\frac{\delta g_{sk}}{\delta x^j}+\frac{\delta g_{js}}{\delta x^k} - \frac{\delta g_{jk}}{\delta x^s}\right),\quad C^{jk}_i=g_{is}C^{sjk}. \label{IV1}
\end{equation}

Also, the Levi-Civita connection $\nabla$ on the Riemannian manifold $(T^*M_0,G)$ is explicitly given with respect to the adapted basis $\{\frac{\delta}{\delta x^i},\frac{\partial}{\partial p_i}\}$, see \cite{A-R1}, Theorem 1, namely
\begin{eqnarray*}
\nabla_{\frac{\delta}{\delta x^i}}\frac{\delta}{\delta x^j}&=&H^k_{ij}\frac{\delta}{\delta x^k}+\frac{1}{2}(R_{ijk}+C_{ijk})\frac{\partial}{\partial p_k},\\
\nabla_{\frac{\delta}{\delta x^{i}}}\frac{\partial}{\partial p_j}&=&\nabla_{\frac{\partial}{\partial p_j}}\frac{\delta}{\delta x^{i}}-N^j_{ih}\frac{\partial}{\partial p_h}\\
&=&-\frac{1}{2}(R_{ish}g^{hj}g^{sk}-2C^{jk}_i)\frac{\delta}{\delta x^k}\\
&&+\frac{1}{2}\left(\frac{\delta g^{jk}}{\delta x^{i}}+\frac{\partial (N_{is}g^{sj})}{\partial p_k}-\frac{\partial (N_{is}g^{sk})}{\partial p_j}\right)g_{kh}\frac{\partial}{\partial p_h},\\
\nabla_{\frac{\partial}{\partial p_i}}\frac{\partial}{\partial p_j}&=&-\frac{1}{2}\left(\frac{\delta g^{ij}}{\delta x^k}+N^{i}_{ks}g^{sj}+N^{j}_{ks}g^{si}\right)g^{kh}\frac{\delta}{\delta x^h}-C^{ij}_k\frac{\partial}{\partial p_k},
\end{eqnarray*}
where $R_{ijk}=\frac{\delta N_{jk}}{\delta x^{i}}-\frac{\delta N_{ik}}{\delta x^j}$, $N^j_{ik}=\frac{\partial N_{ik}}{\partial p_j}$ and $C_{ijk}=-2g_{il}g_{jm}C^{lm}_{k}$.

In our future context we are interested about the Vr\u{a}nceanu connection on the foliated Riemannian manifold $(T^*M_0,G)$ with respect to both vertical foliation $\mathcal{F}_V$ and vertical Liouville-Hamilton foliation $\mathcal{F}_{V_{C^*}}$.

Let $(M,g,\mathcal{F})$ be a Riemannian foliated manifold, with $D$, $D^\perp$, the structural and transversal distributions, respectively. According to \cite{B-F2}, the Vr\u{a}nceanu connection $\nabla^*$ on $(M,f,\mathcal{F})$ is defined by
\begin{equation}
\label{IV2}
\nabla^*_XY=\mathcal{D}\nabla_{\mathcal{D}X}\mathcal{D}Y+\mathcal{D}^\perp \nabla_{\mathcal{D}^\perp X}\mathcal{D}^\perp Y+\mathcal{D}\left[\mathcal{D}^\perp X,\mathcal{D}Y\right]+\mathcal{D}^\perp \left[\mathcal{D}X,\mathcal{D}^\perp Y\right],
\end{equation}
for any $X, Y\in \Gamma(TM)$, where $\nabla$ is the Levi-Civita connection on $(M,g)$, and $\mathcal{D}$ and $\mathcal{D}^\perp$ are the projection morphisms of $\Gamma(TM)$ on $\Gamma(D)$ and $\Gamma(D^\perp)$, respectively.

The linear connection $\nabla^*$ was defined first, using local coordinates, by Vr\u{a}nceanu \cite{VR} on a non-holonomic manifold endowed with a linear connection, where by a non-holonomic manifold we mean a manifold that is endowed with two complementary distributions, at least one of which is non-integrable.

\begin{remark}
We notice that many of the connections that are used in the literature to study foliated manifolds can be related to the Vr\u{a}nceanu connection in one way or the other. For example, Bott connection \cite{T} is the restriction of the Vr\u{a}nceanu connection to transversal distribution. Also, the adapted connection of Reinhart \cite{R} or Vaisman connection (also called second connection) \cite{Va} are the Vr\u{a}nceanu connection on that foliated manifold.
\end{remark}

Using the Levi-Civita connection $\nabla$  on the Riemannian manifold $(T^*M_0,G)$ given above with respect to the adapted basis  $\{\frac{\delta}{\delta x^i},\frac{\partial}{\partial p_i}\}$, we obtain the following local expression for the Vr\u{a}nceanu connection $\nabla'^ *$ on the cotangent bundle of a Cartan space $(M,K)$, endowed with the metric (\ref{I10}) and with vertical foliation, $(T^*M_0,G, \mathcal{F}_V)$:

\[\nabla'^*_{\frac{\partial}{\partial p_j}}\frac{\partial}{\partial p_i}=C'^{ij}_k \frac{\partial}{\partial p_k},\quad \nabla'^*_{\frac{\delta}{\delta x^j}}\frac{\partial}{\partial p_i}=D'^i_{jk} \frac{\partial}{\partial p_k},\]
\[\nabla'^*_{\frac{\partial}{\partial p_j}}\frac{\delta}{\delta x^i}=L'^{kj}_i \frac{\delta}{\delta x^k},\quad \nabla'^*_{\frac{\delta}{\delta x^j}}\frac{\delta}{\delta x^i}=F'^k_{ij} \frac{\delta}{\delta x^k},\]
where the local coefficients of the Vr\u{a}nceanu connection $\nabla'^*$ are given by:
\[C'^{ij}_k =\frac{1}{2}g_{kh}\left(\frac{\partial g^{hi}}{\partial p_j}+\frac{\partial g^{hj}}{\partial p_i}-\frac{\partial g^{ij}}{\partial p_h}\right)=-g_{kh}C^{hij}=-C^{ij}_k,\]
\begin{equation}
\label{IV3}
D'^i_{jk}=-\frac{\partial N_{jk}}{\partial p_i},\quad L'^{kj}_i=0,\quad F'^k_{ij}=H^k_{ij}.
\end{equation}
The only nonzero local components of the torsion tensor field $T'^*$ of the Vr\u{a}nceanu connection are
\[T'^*\left(\frac{\delta}{\delta x^i},\frac{\delta}{\delta x^j}\right)=T'^*_{ijk}\frac{\partial}{\partial p_k}=\left(\frac{\delta N_{ik}}{\delta x^j}-\frac{\delta N_{jk}}{\delta x^i}\right)\frac{\partial}{\partial p_k}.\]
Hence, it results the following:
\begin{proposition}
The Vr\u{a}nceanu connection $\nabla^{\prime*}$ on $(T^*M_0,G, \mathcal{F}_V)$ is torsion-free if and only if the canonical nonlinear connection $N$ is integrable.
\end{proposition}

Related to Vr\u{a}nceanu connection (Vaisman connection) there is the notion of \textit{Reinhart space}. A Riemannian foliated manifold $(M,g, F)$ is called a Reinhart space iff 
\[(\nabla^*_Xg)(Y,Z)=0,\quad \forall X\in \Gamma(T\mathcal{F}),Y,Z\in \Gamma(T^{\bot}\mathcal{F}).\]

\begin{proposition}
Let $(M,K)$ be a Cartan space. The foliated manifold $(T^*M_0,G,\mathcal{F}_V$) is a Reinhart space if and only if the metric $g=(g^{ij})$ is Riemannian on $M$.
\end{proposition}
\begin{proof}
Since $\nabla'^*_{\frac{\partial}{\partial p_j}}\frac{\delta}{\delta  x^i}=0$, we have
\[\left(\nabla'^*_{\frac{\partial}{\partial p_k}}G\right)\left(\frac{\delta}{\delta  x^i},\frac{\delta}{\delta  x^j}\right)=\frac{\partial g_{ij}}{\partial p_k},\]
which vanishes iff $g_{ij}$ does not depends by the momenta $p_k$, for all indices $i,j,k$. It follows that the condition $\left(\nabla'^*_XG\right)(Y,Z)=0,\quad \forall X\in \Gamma(V(T^*M_0)),Y,Z\in \Gamma(H(T^*M_0))$ is equivalent to $g^{ij}=g^{ij}(x)$, so it is a Riemannian metric on $M$.
\end{proof}

Now, we consider the Riemannian manifold $(T^*M_0,G)$ endowed with the vertical Liouville-Hamilton foliation $\mathcal{F}_{V_{C^*}}$. As we already seen the structural bundle of this foliation is given by $V_{C^*}$ and the transversal bundle is given by $\{C^*\}\oplus H(T^*M_0)$. We denote by $\it{v}_{C^*}$ the projector of $T(T^*M_0)$ onto $V_{C^*}$ and by $q$ the projector of  $T(T^*M_0)$ onto $\{C^*\}\oplus H(T^*M_0)$. If we denote by $\nabla^{\prime\prime*}$ the Vr\u{a}nceanu connection on the foliated Riemannian manifold $(T^*M_0,G,\mathcal{F}_{V_{C^*}})$, by using the relation (\ref{IV2}) with $\mathcal{D}=\it{v}_{C^*}$ and $\mathcal{D}^\perp=q$, and the second formula of \eqref{III7}, we obtain
\begin{equation}
\label{IV4}
\nabla^{\prime\prime*}_{C^*}\frac{\overline{\partial}}{\overline{\partial}p_i}=-\frac{\overline{\partial}}{\overline{\partial}p_i},\quad \nabla^{\prime\prime*}_{\frac{\overline{\partial}}{\overline{\partial}p_i}}C^*=0,
\end{equation}
where $V_{C^*}$ is spanned by $\{\frac{\overline{\partial}}{\overline{\partial}p_i}=P_j^i\frac{\partial}{\partial p_j}\}$ and $\{C^*\}\oplus H(T^*M_0)$ is spanned by $\{C^*,\frac{\delta}{\delta x^i}\}$, $i=1,\ldots,n$. 

The whole expression of this Vr\u{a}nceanu connection in the adapted basis $\left\{\frac{\overline{\partial}}{\overline{\partial}p_i},C^*,\frac{\delta}{\delta x^{i}}\right\}$ with respect to the decomposition $T(T^*M_0)=V_{C^*}\oplus\{C^*\}\oplus H(T^*M_0)$ is given by:
\begin{eqnarray*}
\nabla^{\prime\prime*}_{C^*}C^*&=&C^*,\\
\nabla^{\prime\prime*}_{C^*}{\frac{\delta}{\delta x^{i}}}&=&\frac{p_jp^h}{K^2}\left(\frac{1}{2}\left(\frac{\delta g^{jk}}{\delta x^{i}}+\frac{\partial(N_{is}g^{sj})}{\partial p_k}-\frac{\partial(N_{is}g^{sk})}{\partial p_j}\right)g_{kh}+N^j_{ih}\right)C^*\\
&&-\frac{p_j}{2}(R_{ish}g^{hj}g^{sk}-2C^{jk}_i)\frac{\delta}{\delta x^k},\\
\nabla^{\prime\prime*}_{\frac{\overline{\partial}}{\overline{\partial}p_i}}\frac{\delta}{\delta x^j}&=&\left(P^{i}_lN^{l}_{jk}-\frac{\delta P^{i}_k}{\delta x^j}\right)\frac{p^k}{K^2}C^*,\\
\nabla^{\prime\prime*}_{\frac{\overline{\partial}}{\overline{\partial}p_i}}\frac{\overline{\partial}}{\overline{\partial}p_j}&=&\left(P^{i}_k\frac{\partial P^j_l}{\partial p_k}-P^{i}_kP^j_mC^{km}_l\right)\frac{\overline{\partial}}{\overline{\partial}p_l}\\
&=&-\frac{1}{K^2}\left(P^{i}_lp^j+h^{ij}p_l+K^2P^{i}_kP^j_mC^{km}_l\right)\frac{\overline{\partial}}{\overline{\partial}p_l},\\
\nabla^{\prime\prime*}_{\frac{\delta}{\delta x^{i}}}\frac{\overline{\partial}}{\overline{\partial}p_j}&=&\left(\frac{\delta P^j_k}{\delta x^{i}}-P^j_l N^l_{ik}\right)\frac{\overline{\partial}}{\overline{\partial}p_k},\\
\nabla^{\prime\prime*}_{\frac{\delta}{\delta x^{i}}}\frac{\delta}{\delta x^{j}}&=&H^k_{ij}\frac{\delta}{\delta x^{k}}+\frac{p^k}{2K^2}\left(R_{ijk}+C_{ijk}\right)C^*,\\
\nabla^{\prime\prime*}_{\frac{\delta}{\delta x^{i}}}C^*&=&\frac{1}{K^2}\left(N_{ik}p^k+\frac{p_j}{2}\left(\frac{\delta g^{jk}}{\delta x^{i}}+\frac{\partial(N_{is}g^{sj})}{\partial p_k}-\frac{\partial(N_{is}g^{sk})}{\partial p_j}\right)g_{kh}p^h\right)C^*\\
&&-\frac{p_j}{2}(R_{ish}g^{hj}g^{sk}-2C^{jk}_i)\frac{\delta}{\delta x^k},
\end{eqnarray*}
but in our future study we have need only the relations \eqref{IV4}.

In the end of this section, we notice that for every fixed point $x_0\in M$, the leaf $F_V=T^*_{x_0}M$ of the vertical foliation $\mathcal{F}_V$ is also a Riemannian  manifold, foliated by $V_{C^*}$ with transversal distribution $\{C^*\}$. If we consider the restriction of the Vr\u{a}nceanu connection $\nabla^{\prime\prime*}$ and of the metric $G$ to vertical vector fields, using \eqref{IV4}, we obtain 

	\[\left(\nabla^{\prime\prime*} _{\frac{\overline{\partial}}{\overline{\partial} p_i}}G_{F_V}\right)(C^*,C^*)=\frac{\overline{\partial}K^2}{\overline{\partial} p_i}-2G_{F_V}(\nabla^{\prime\prime*} _{\frac{\overline{\partial}}{\overline{\partial} p_i}}C^*,C^*)=2K\frac{\overline{\partial}K}{\overline{\partial} p_i}=0,
\]
for every $i=1,\ldots,n$ which leads to

\begin{proposition}
For every fixed point $x_0\in M$, $(T^*_{x_0}M,G_{F_V},\mathcal{F}_{V_{C^*}})$ is a Reinhart space. 
\end{proposition}
\section{Subfoliations in the cotangent manifold of a Cartan space}
\setcounter{equation}{0}

In this section, following \cite{C-M}, we briefly recall the notion of a $(q_1,q_2)$--codimensional subfoliation on a manifold and we identify a $(n,2n-1)$--codimensional subfoliation $(\mathcal{F}_V,\mathcal{F}_{C^*})$ on the cotangent manifold $T^*M_0$ of a Cartan space $(M,K)$, where $\mathcal{F}_V$ is the vertical foliation and $\mathcal{F}_{C^*}$ is the line foliation spanned by the vertical Liouville-Hamilton vector field $C^*$. Firstly, we make a general approach about basic connections on the normal bundles related to this subfoliation and next a triple of adapted basic connections with respect to this subfoliation is given.

\begin{definition}
Let $M$ be a $n$--dimensional manifold and $TM$ its tangent bundle. A \textit{$(q_1,q_2)$--codimensional subfoliation} on $M$ is a couple $(F_1,F_2)$ of integrable subbundles $F_k$ of $TM$ of dimension $n-q_k$, $k=1,2$ and $F_2$ being at the same time a subbundle of $F_1$.
\end{definition}
For a subfoliation $(F_1,F_2)$, its normal bundle is defined as $Q(F_1,F_2)=QF_{21}\oplus QF_1$, where $QF_{21}$ is the quotient bundle $F_1/F_2$ and $QF_1$ is the usual normal bundle of $F_1$. So, an exact sequence of vector bundles
\begin{equation}
0\longrightarrow QF_{21}\stackrel{i}{\longrightarrow} QF_2\stackrel{\pi}{\longrightarrow}QF_1\longrightarrow0
\label{V1}
\end{equation}
appears in a canonical way.

Also if we consider the canonical exact sequence associated to the foliation given by an integrable subbundle $F$, namely
\begin{displaymath}
0\longrightarrow F\stackrel{i_F}{\longrightarrow}TM\stackrel{\pi_F}{\longrightarrow}QF\longrightarrow0
\end{displaymath}
then we recall that a connection $D:\Gamma(TM)\times\Gamma(QF)\rightarrow\Gamma(QF)$ on the normal bundle $QF$ is said to be \textit{basic} if 
\begin{equation}
D_XY=\pi_F[X,\widetilde{Y}]
\label{V2}
\end{equation} 
for any $X\in\Gamma(F)$, $\widetilde{Y}\in\Gamma(TM)$ such that $\pi_F(\widetilde{Y})=Y$.

Similarly, for a $(q_1,q_2)$--subfoliation $(F_1,F_2)$ we can consider the following exact sequence of vector bundles
\begin{equation}
0\longrightarrow F_2\stackrel{i_0}{\longrightarrow}F_1\stackrel{\pi_0}{\longrightarrow}QF_{21}\longrightarrow0
\label{V3}
\end{equation}
and according to \cite{C-M} a connection $\nabla$ on $QF_{21}$ is said to be basic with respect to the subfoliation $(F_1,F_2)$ if
\begin{equation}
D_XY=\pi_0[X,\widetilde{Y}]
\label{V4}
\end{equation}
for any $X\in\Gamma(F_2)$ and $\widetilde{Y}\in\Gamma(F_1)$ such that $\pi_0(\widetilde{Y})=Y$.

\subsection{A $(n,2n-1)$--codimensional subfoliation $(\mathcal{F}_V,\mathcal{F}_{C^*})$ of $(T^*M_0,G)$}
Taking into account the discussion from the previous section, for an $n$--dimensional Cartan space $(M,K)$, we have on the $2n$--dimensional cotangent manifold $T^*M_0$ a $(n,2n-1)$--codimensional subfoliation $(\mathcal{F}_V,\mathcal{F}_{C^*})$. We also notice that the metric structure $G$ on $T^*M_0$ given by \eqref{I10} is compatible with the subfoliated structure, that is 
\begin{displaymath}
Q\mathcal{F}_V\cong H(T^*M_0),\,Q\mathcal{F}_{C^*}\cong \{C^*\}^{\perp},\,V(T^*M_0)/\{C^*\}\cong V_{C^*}.
\end{displaymath} 
Let us consider the following exact sequences associated to the subfoliation $(\mathcal{F}_V,\mathcal{F}_{C^*})$ 
\begin{displaymath}
0 \longrightarrow \{C^*\} \stackrel{i_0}{\longrightarrow} V(T^*M_0) \stackrel{\pi_0}{\longrightarrow}{V_{C^*}} \longrightarrow 0,
\end{displaymath}
and to foliations $\mathcal{F}_V$ and $\mathcal{F}_{C^*}$, respectively
\begin{displaymath}
0 \longrightarrow V(T^*M_0) \stackrel{i_1}{\longrightarrow} T(T^*M_0) \stackrel{\pi_1}{\longrightarrow}H(T^*M_0) \longrightarrow 0,
\end{displaymath}
\begin{displaymath}
0 \longrightarrow \{C^*\} \stackrel{i_2}{\longrightarrow} T(T^*M_0) \stackrel{\pi_2}{\longrightarrow}\{C^*\}^{\perp} \longrightarrow 0,
\end{displaymath}
where $i_0,i_1,i_2$, $\pi_0,\pi_1,\pi_2$ are the canonical inclusions and projections, respectively. 

A triple $(D^1, D^2,D)$ of basic connections on  normal bundles $V_{C^*}$ of subfoliation $(\mathcal{F}_{V},\mathcal{F}_{C^*})$, $H(T^*M_0)$ of vertical foliation $\mathcal{F}_V$, and $\{C^*\}^{\perp}$ of line foliation $\mathcal{F}_{C^*}$, respectively, is called (according to the terminology used in \cite{C-M}) \textit{adapted} to the subfoliation $(\mathcal{F}_{V},\mathcal{F}_{C^*})$.

Our goal is to determine such a triple of connections, adapted to this subfoliation.

By \eqref{V4} a connection $D^1$ on $V_{C^*}$ is basic with respect to the subfoliation $(\mathcal{F}_V,\mathcal{F}_{C^*})$ if 
\begin{equation}
D^1_XZ=\pi_0[X,\widetilde{Z}],\, \forall\, X\in \Gamma(\{C^*\}),\, \forall\, \widetilde{Z}\in \Gamma(V(T^*M_0)),\,\pi_0(\widetilde{Z})=Z.
\label{V5}
\end{equation}
\begin{proposition}
\label{p5.1}
A connection $D^1$ on $V_{C^*}$ is basic with respect to the subfoliation $(\mathcal{F}_{V},\mathcal{F}_{C^*})$ if and only if 
\begin{displaymath}
D^1_{C^*}Z=[C^*,Z],\, \forall\, Z\in \Gamma(V_{C^*}).
\end{displaymath}
\end{proposition}
\begin{proof}
Let $D^1$ be a connection on $V_{C^*}$ such that $D^1_{C^*}Z=[C^*,Z]$. Let $X\in \Gamma(\{C^*\})$ be a section in the structural bundle of the line foliation $\mathcal{F}_{C^*}$, so its form is $X=aC^*$, with $a$ a differentiable function on $T^*M_0$. An arbitrary vertical vector field $\widetilde{Z}$ which projects into $Z\in V_{C^*}$ is in the form
\begin{displaymath}
\widetilde{Z}=Z+bC^*
\end{displaymath}
with $b$ a differentiable function on $T^*M_0$. 

We have
\begin{eqnarray*}
[X,\widetilde{Z}]&=&[aC^*,Z+bC^*]\\
&=&a[C^*,Z]+(aC^*(b)-bC^*(a)-Z(a))C^*.
\end{eqnarray*}
According to the second relation from \eqref{III7} for any $Z=Z_i\frac{\overline{\partial}}{\overline{\partial}p_i}\in\Gamma\left(V_{C^*}\right)$, we have
\begin{displaymath}
[C^*,Z]=(C^*(Z_i)-Z_i)\frac{\overline{\partial}}{\overline{\partial}p_i}\in\Gamma\left( V_{C^*}\right),
\end{displaymath}
so $\pi_0[X,\widetilde{Z}]=a[C^*,Z]$. We also have $D^1_XZ=aD^1_{C^*}Z=a[C^*,Z]=\pi_0[X,\widetilde{Z}]$, hence $D^1$ is a basic connection on $V_{C^*}$.

Conversely, by the second relation from \eqref{III7}, in the adapted basis $\left\{\frac{\overline{\partial}}{\overline{\partial}p_i},C^*\right\}$ in $V(T^*M_0)$, every basic connection $D^1$ on $V_{C^*}$ is locally satisfying
\begin{equation}
D^1_{C^*}\frac{\overline{\partial}}{\overline{\partial}p_i}=-\frac{\overline{\partial}}{\overline{\partial}p_i},
\label{V6}
\end{equation}
for any $i=1,\ldots,n$.
  
Now, if \eqref{V6} is satsfied, then  
\begin{displaymath}
D^1_{C^*}Z=C^*(Z_i)\frac{\overline{\partial}}{\overline{\partial}p_i}+Z_iD^1_{C^*}\frac{\overline{\partial}}{\overline{\partial}p_i}=C^*(Z_i)\frac{\overline{\partial}}{\overline{\partial}p_i}-Z_i\frac{\overline{\partial}}{\overline{\partial}p_i}.
\end{displaymath}
Hence the condition \eqref{V6} is equivalent with $D^1_{C^*}Z=[C^*,Z],\quad \forall Z\in \Gamma\left(V_{C^*}\right)$.
\end{proof}

Thus, we have obtained the following locally characterisation:
\begin{proposition}
\label{p5.2}
A connection $D^1$ on $V_{C^*}$ is basic if and only if in an adapted local chart the relation \eqref{V6} holds. 
\end{proposition}

Now, by \eqref{V2}, a connection $D^2$ on $H(T^*M_0)$ is  basic with respect to the vertical foliation $\mathcal{F}_V$ if 
\begin{equation}
D^2_XY=\pi_1[X,\widetilde{Y}],
\label{V7}
\end{equation}
for any $X\in \Gamma(V(T^*M_0))$ and $\widetilde{Y}\in\Gamma(T(T^*M_0))$ such that $\pi_1(\widetilde{Y})=Y$.
\begin{proposition}
\label{p5.3}
A connection $D^2$ on $H(T^*M_0)$ is basic if and only if in an adapted local frame $\left\{\frac{\delta}{\delta x^i},\frac{\partial}{\partial p_i}\right\}$ on $T(T^*M_0)$ we have 
\begin{displaymath}
D^2_{\frac{\partial}{\partial p_j}}\frac{\delta}{\delta x^i}=0,
\end{displaymath}
for any $i,j=1,\ldots,n$.
\end{proposition}
\begin{proof} 
If $D^2$ is a basic connection with respect to the vertical foliation, then by definition it results 
\begin{displaymath}
D^2_{\frac{\partial}{\partial p_j}}\frac{\delta}{\delta x^i}=\pi_1\left[\frac{\delta}{\delta x^i},\frac{\partial}{\partial p_j}\right]=\pi_1\left(-\frac{\partial N_{ki}}{\partial p_j}\frac{\partial}{\partial p_k}\right)=0.
\end{displaymath}
Conversely, let $D^2:T(T^*M_0)\times H(T^*M_0)\rightarrow H(T^*M_0)$ be a connection on $H(T^*M_0)$ which locally satisfies 
$
D^2_{\frac{\partial}{\partial p_j}}\frac{\delta}{\delta x^i}=0,
$
for any $i,j=1,\ldots,n$.

An arbitrary vertical vector field $X$ has local expression $X=X_i\frac{\partial}{\partial p_i}$ and a vector field $\widetilde{Y}$ whose horizontal projection is $Y=Y_h^i\frac{\delta}{\delta x^i}$ is by the form $\widetilde{Y}=Y+Y_i^v\frac{\partial}{\partial p_i}$. 

We calculate
\begin{displaymath}
D^2_XY=X_iD^2_{\frac{\partial}{\partial p_i}}\left(Y_h^j\frac{\delta}{\delta  x^j}\right)=X_i\frac{\partial Y_h^j}{\partial p_i}\frac{\delta}{\delta  x^j},
\end{displaymath}
\begin{displaymath}
\left[X,\widetilde{Y}\right]=X_i\frac{\partial Y_h^j}{\partial p_i}\frac{\delta}{\delta  x^j}+\left(X_i\frac{\partial Y_v^j}{\partial p_i}-Y_h^i\frac{\delta X_j}{\delta x^i}\right)\frac{\partial}{\partial p_j}+X_iY^j_h\left[\frac{\partial}{\partial p_i},\frac{\delta}{\delta  x^j}\right],
\end{displaymath}
hence the relation \eqref{V7} is verified, since $\left[\frac{\partial}{\partial p_i},\frac{\delta}{\delta  x^j}\right]\in \Gamma(V(T^*M_0))$. So, $D^2$ is a basic connection with respect to the vertical foliation $\mathcal{F}_V$.
\end{proof}

Also, by \eqref{V2}, a connection $D$ on $\{C^*\}^{\perp}$ is basic with respect to the line foliation $\mathcal{F}_{C^*}$ if 
\begin{equation}
D_XY=\pi_2[X,\widetilde{Y}],
\label{V8}
\end{equation}
for any $X\in \Gamma(\{C^*\})$ and $\widetilde{Y}\in \Gamma(T(T^*M_0))$ such that $\pi_2(\widetilde{Y})=Y$.

We have the following locally characterisation of a basic connection on $\{C^*\}^{\perp}$:

\begin{proposition}
\label{p5.4}
A connection $D$ on $\{C^*\}^{\perp}$ is basic with respect to the line foliation $\mathcal{F}_{C^*}$ if and only if in an adapted local frame $\left\{\frac{\delta}{\delta x^i},\frac{\overline{\partial}}{\overline{\partial}p_i}\right\}$ on $\{C^*\}^{\perp}$ we have 
\begin{equation}
D_{C^*}\frac{\delta}{\delta x^i}=0,\quad D_{C^*}\frac{\overline{\partial}}{\overline{\partial}p_i}=-\frac{\overline{\partial}}{\overline{\partial}p_i},
\label{V9}
\end{equation}
for any $i=1,\ldots,n$.
\end{proposition}
\begin{proof} Let $D$ be a basic connection on $\{C^*\}^{\perp}$. Since $\{C^*\}^{\perp}$ is locally generated by $\left\{\frac{\delta}{\delta x^i},\frac{\overline{\partial}}{\overline{\partial}p_i}\right\}$, the condition \eqref{V8} give us the following relations:
\begin{displaymath}
D_{C^*}\frac{\delta}{\delta x^i}=\pi_2\left[C^*,\frac{\delta}{\delta x^i}+fC^*\right]=\pi_2\left((-N_{ji}+C^*(N_{ji}))\frac{\partial}{\partial p_j}+C^*(f)C^*\right)=0,
\end{displaymath}
since $C^*(N_{ji})=N_{ji}$, by the homogeneity of degree 1 in momenta of functions $N_{ji}$, see \cite{M-H-S}, and
\begin{displaymath}
D_{C^*}\frac{\overline{\partial}}{\overline{\partial}p_i}=\pi_2\left[C^*,\frac{\overline{\partial}}{\overline{\partial}p_i}+fC^*\right]=\pi_2\left(-\frac{\overline{\partial}}{\overline{\partial}p_i}+C^*(f)C^*\right)=-\frac{\overline{\partial}}{\overline{\partial}p_i}.
\end{displaymath}
Conversely, let us consider $D:T(T^*M_0)\times \{C^*\}^{\perp}\rightarrow \{C^*\}^{\perp}$ be a connection on $\{C^*\}^{\perp}$ which locally satisfies \eqref{V9}.

An arbitrary vector field $Y\in \Gamma\left(\{C^*\}^{\perp}\right)$ is locally given by $Y=Y_h^i\frac{\delta}{\delta x^i}+Y_i\frac{\overline{\partial}}{\overline{\partial}p_i}$ and a vector field $\widetilde{Y}\in \Gamma(T(T^*M_0))$ which projects by $\pi_2$ in $Y$ is $\widetilde{Y}=fC^*+Y$. For an arbitrary vector field $X=aC^*\in \Gamma(\{C^*\})$, we calculate

\begin{eqnarray*}
D_{X}Y&=&D_{aC^*}\left(Y_h^i\frac{\delta}{\delta x^i}+Y_i\frac{\overline{\partial}}{\overline{\partial}p_i}\right)\\
&=&aC^*(Y_h^i)\frac{\delta}{\delta x^i}+aY_h^iD_{C^*}\frac{\delta}{\delta x^i}+aC^*(Y_i)\frac{\overline{\partial}}{\overline{\partial}p_i}+aY_iD_{C^*}\frac{\overline{\partial}}{\overline{\partial}p_i}\\
&=&aC^*(Y_h^i)\frac{\delta}{\delta x^i}+a(C^*(Y_i)-Y_i)\frac{\overline{\partial}}{\overline{\partial}p_i},
\end{eqnarray*}
and
\begin{eqnarray*}
\pi_2[X,\widetilde{Y}]&=&\pi_2\left(aC^*(f)C^*-fC^*(a)C^*+aC^*(Y_h^i)\frac{\delta}{\delta x^i}+aY^i_h\left[C^*,\frac{\delta}{\delta x^i}\right]\right)\\
&&+\pi_2\left(a(C^*(Y_i)-Y_i)\frac{\overline{\partial}}{\overline{\partial}p_i}-Y^i_h\frac{\delta a}{\delta x^i}C^*-Y_i\frac{\overline{\partial}a}{\overline{\partial} p_i}C^*\right)\\
&=&aC^*(Y_h^i)\frac{\delta}{\delta x^i}+a(C^*(Y_i)-Y_i)\frac{\overline{\partial}}{\overline{\partial}p_i},
\end{eqnarray*}
since $\left[C^*,\frac{\delta}{\delta x^i}\right]=0$. Thus, we have obtained that the connection $D$ is a basic one.
\end{proof} 
 
\subsection{A triple of adapted basic connections to subfoliation $(\mathcal{F}_V,\mathcal{F}_{C^*})$}The restriction of Vr\u{a}nceanu connection $\nabla^{\prime*}$ from subsection 4.1, to $\Gamma(T(T^*M_0))\times \Gamma(H(T^*M_0))$, is a connection on $H(T^*M_0)$ denoted by $\overline{\nabla}^{\prime*}$, which satisfies the conditions from Proposition \ref{p5.3}, (see relations \eqref{IV3}), so it is a basic connection on $H(T^*M_0)$ with respect to vertical foliation $\mathcal{F}_V$. This fact follows also from Remark 3.1 since the restriction of Vr\u{a}nceanu connection $\nabla^{\prime*}$ on $H(T^*M_0)$ is the Bott connection on $H(T^*M_0)$.

Also, the restriction of the Vr\u{a}nceanu connection $\nabla^{\prime\prime*}$ from the third section to $V_{C^*}$,
induces a connection $\overline{\nabla}^{\prime\prime*}$ on $V_{C^*}$, which satisfies the conditions from Proposition \ref{p5.2}, (see relations \eqref{IV4}), so it is a basic connection on $V_{C^*}$ with respect to the subfoliation $(\mathcal{F}_V,\mathcal{F}_{C^*})$.

Hence we have the basic connection $\overline{\nabla}^{\prime*}$ on $H(T^*M_0)$ with respect to vertical foliation $\mathcal{F}_V$, and the basic connection $\overline{\nabla}^{\prime\prime*}$ on $V_{C^*}$ with respect to the subfoliation $(\mathcal{F}_V,\mathcal{F}_{C^*})$. Following \cite{C-M}, we can build now a connection on $\{C^*\}^{\perp}$ as follows:
\begin{displaymath}
\overline{\nabla}:T(T^*M_0)\times \{C^*\}^{\perp}\rightarrow \{C^*\}^{\perp}\,,\,\overline{\nabla}_XZ=\overline{\nabla}^{\prime*}_{X}Z^h+\overline{\nabla}^{\prime\prime*}_XZ',
\end{displaymath}
for any $Z=Z^h+Z'\in\Gamma\left( \{C^*\}^{\perp}\right)=\Gamma\left(H(T^*M_0)\right)\oplus\Gamma\left(V_{C^*}\right)$ and $X\in \Gamma(T(T^*M_0))$.

By direct calculus we have
\begin{displaymath}
\overline{\nabla}_{C^*}\frac{\delta}{\delta x^i}=\overline{\nabla}^{\prime*}_{C^*}\frac{\delta}{\delta x^i}=0, \quad \overline{\nabla}_{C^*}\frac{\overline{\partial}}{\overline{\partial}p_i}=\overline{\nabla}^{\prime\prime*}_{C^*}\frac{\overline{\partial}}{\overline{\partial}p_i}=-\frac{\overline{\partial}}{\overline{\partial}p_i},
\end{displaymath}
so the connection $\overline{\nabla}$ satisfies conditions from Proposition \ref{p5.4}, hence it is a basic connection on $\{C^*\}^{\perp}$ with respect to the line foliation $\mathcal{F}_{C^*}$. 

The whole expression of the adapted basic connection $\overline{\nabla}$ is given by
\begin{eqnarray*}
\overline{\nabla}_{\frac{\overline{\partial}}{\overline{\partial}p_i}}\frac{\delta}{\delta x^j}&=&\nabla^{\prime*}_{\frac{\overline{\partial}}{\overline{\partial}p_i}}\frac{\delta}{\delta x^j}=0\,,\,\overline{\nabla}_{\frac{\delta}{\delta x^i}}\frac{\delta}{\delta x^j}=\nabla^{\prime*}_{\frac{\delta}{\delta x^i}}\frac{\delta}{\delta x^j}=H^k_{ij}\frac{\delta}{\delta x^k},\\
\overline{\nabla}_{\frac{\overline{\partial}}{\overline{\partial}p_i}}\frac{\overline{\partial}}{\overline{\partial}p_j}&=&\nabla^{\prime\prime*}_{\frac{\overline{\partial}}{\overline{\partial}p_i}}\frac{\overline{\partial}}{\overline{\partial}p_j}=-\frac{1}{K^2}\left(P^{i}_lp^j+h^{ij}p_l+K^2P^{i}_kP^j_mC^{km}_l\right)\frac{\overline{\partial}}{\overline{\partial}p_l},\\
\overline{\nabla}_{\frac{\delta}{\delta x^{i}}}\frac{\overline{\partial}}{\overline{\partial}p_j}&=&\nabla^{\prime\prime*}_{\frac{\delta}{\delta x^{i}}}\frac{\overline{\partial}}{\overline{\partial}p_j}=\left(\frac{\delta P^j_k}{\delta x^{i}}-P^j_l N^l_{ik}\right)\frac{\overline{\partial}}{\overline{\partial}p_k}.
\end{eqnarray*}

Now, for the $(n,2n-1)$--codimensional subfoliation $(\mathcal{F}_V,\mathcal{F}_{C^*})$ we can consider the exact sequence \eqref{V1}, namely
\begin{displaymath}
0\longrightarrow V_{C^*}\stackrel{i}{\longrightarrow}\{C^*\}^{\perp}\stackrel{\pi}{\longrightarrow}H(T^*M_0)\longrightarrow0,
\end{displaymath}
and the general theory of $(q_1,q_2)$--codimensional foliations, see \cite{C-M}, leads to the following results:
\begin{proposition}
\label{p5.5}
For the triple basic connections $(\overline{\nabla}^{\prime\prime*},\overline{\nabla}^{\prime*},\overline{\nabla})$ on $V_{C^*}$, $H(T^*M_0)$ and $\{C^*\}^{\perp}$, respectively, we have
\begin{displaymath}
i\left(\overline{\nabla}^{\prime\prime*}_XY\right)=\overline{\nabla}_Xi(Y)\,,\,\pi\left(\overline{\nabla}_XZ\right)=\overline{\nabla}^{\prime*}_{X}\pi(Z)\,,\,\,\forall\,Y\in\Gamma\left(V_{C^*}\right),\,Z\in\Gamma\left(\{C^*\}^{\perp}\right).
\end{displaymath}
\end{proposition}
\begin{proposition}
\label{p5.6}
The triple of basic connections $(\overline{\nabla}^{\prime\prime*},\overline{\nabla}^{\prime*},\overline{\nabla})$ is adapted to the subfoliation $(\mathcal{F}_V,\mathcal{F}_{C^*})$ of $(T^*M_0,G)$.
\end{proposition}

As in the case of general theory for subfoliations, \cite{C-M}, a possible application of the adapted basic connection constructed here for our subfoliation may be given in the study of characteristic classes of the $(n,2n-1)$--codimensional  subfoliation  $(\mathcal{F}_V,\mathcal{F}_{C^*})$ and if this study can be related with some geometrical  or topological properties of the base Cartan space $(M,K)$.   In the next section we relate some notions about cohomology of the indicatrix (which is in fact the foliated cohomology of the fundamental foliation defined by the distribution $\{C^*\}^\perp$) with a Poincar\'{e}-Cartan $2$--form on $(M,K)$.   

\section{On the cohomology of the $c$-indicatrix cotangent bundle}
\setcounter{equation}{0}
In this section, firstly, using the vertical Liouville-Hamilton vector field $C^*$ and the natural almost complex structure $J$ on $T^*M_0$, we give an adapted basis in $T(T^*M_0)$. Next we prove that the $c$--indicatrix cotangent bundle $I(M,K)(c)$ of $(M,K)$ is a $CR$--submanifold of the almost K\"{a}hlerian manifold $(T^*M_0,G,J)$ and we study some cohomological properties of $I(M,K)(c)$ in relation with classical cohomology of $CR$--submanifolds, \cite{C1}.

For the natural almost complex structure $J$ on $T(T^*M_0)$, we consider now the new local vector field frame in $T(T^*M_0)$ as $\left\{\frac{\overline{\delta}}{\overline{\delta} x_i}, \xi^*,\frac{\overline{\partial}}{\overline{\partial} p_i},C^*\right\}$, where
\begin{equation}
\xi^*=J\left(C^*\right)=p^i\frac{\delta}{\delta x^i}
\label{III9}
\end{equation} 
and
\begin{equation}
\frac{\overline{\delta}}{\overline{\delta} x_i}=J\left(\frac{\overline{\partial}}{\overline{\partial} p_i}\right)=P^i_jg^{jk}\frac{\delta}{\delta x^k}=h^{ki}\frac{\delta}{\delta x^k}=g^{ki}\frac{\delta}{\delta x^k}-\frac{p^i}{K^2}\xi^*.
\label{r4}
\end{equation} 

As in the second section it follows that $H_{\xi^*}:={\rm Span}\left\{\frac{\overline{\delta}}{\overline{\delta} x_1},\ldots,\frac{\overline{\delta}}{\overline{\delta} x_n}\right\}$ is an $(n-1)$--dimensional horizontal sub-distribution, orthogonal to $\left\{\xi^*\right\}$ in $H(T^*M_0)$, where $\left\{\xi^*\right\}$ is the line distribution spanned by the horizontal Liouville-Hamilton vector field $\xi^*$.

Since the vertical Liouville-Hamilton vector field $C^*$ is orthogonal to the level hypersurfaces of the fundamental function $K$, the vector fields $\left\{\frac{\overline{\delta}}{\overline{\delta} x_i}, \xi^*,\frac{\overline{\partial}}{\overline{\partial} p_i}\right\}$ are tangent to these hypersurfaces in $T^*M_0$, so they generate the distribution $\{C^*\}^{\perp}$ which is the orthogonal complement of $\{C^*\}$ in $T(T^*M_0)$. The vertical indicatrix (Liouville) distribution $V_{C^*}$ is locally generated by $\left\{\frac{\overline{\partial}}{\overline{\partial} p_i}\right\}$,  and  the vertical foliation has the structural bundle locally generated by $\left\{\frac{\overline{\partial}}{\overline{\partial} p_i},C^*\right\}$, $i=1,\ldots,n$. Also, we have the decomposition
\begin{equation}
\{C^*\}^{\perp}=\left\{\xi^*\right\}\oplus H_{\xi^*}\oplus V_{C^*}.
\label{III10}
\end{equation}
For any $c>0$, we consider now the \textit{$c$--indicatrix cotangent bundle} over $M$, given by 
\begin{displaymath}
I(M,K)(c)=\bigcup_{x\in M}I_x(M,K)(c)\,\,,\,\,I_x(M,K)(c)=\{p\in T_x^*M_0\,:\,K(x,p)=c\}.
\end{displaymath}

According to \cite{Be1, B-F}, if $(\widetilde{N},\widetilde{g},\widetilde{J})$ is an (almost) K\"{a}hler manifold, where $\widetilde{g}$ is the Riemannian metric and $\widetilde{J}$ is the (almost) complex structure on $\widetilde{N}$, then $N$ is a \textit{$CR$-submanifold} of $\widetilde{N}$ if $N$ admits two complementary orthogonal distributions $D$ and $D^{\perp}$ such that
\begin{enumerate}
\item[i)] $D$ is $\widetilde{J}$--invariant, i.e., $\widetilde{J}(D)\subset D$;
\item[ii)] $D^{\perp}$ is $\widetilde{J}$--anti-invariant, i.e., $\widetilde{J}\left(D^{\perp}\right)\subset (TN)^{\perp}$, where $(TN)^\perp$ is the orthogonal complement of $TN$ in $T\widetilde{N}$.
\end{enumerate}
$D$ is called \textit{maximal complex (holomorphic)} distribution of $N$ and $D^{\perp}$ is called \textit{totally real} distribution of $N$.

We have
\begin{proposition}
Let $i:I(M,K)(c)\hookrightarrow T^*M_0$ be the imersion of $I(M,K)(c)$ in $T^*M_0$. Then $I(M,K)(c)$ is a $CR$--submanifold of $T^*M_0$ with holomorphic distribution given by $D=H_{\xi^*}\oplus V_{C^*}$ and the totally real distribution given by $D^{\perp}=\{\xi^*\}$.
\end{proposition}
\begin{proof}
We have that $\{C^*\}^{\perp}=\left\{\xi^*\right\}\oplus H_{\xi^*}\oplus V_{C^*}$ is the tangent bundle of $I(M,K)(c)$. Taking into account the behaviour of the almost complex structure $J$ of $(T^*M_0,G)$ we have
\begin{displaymath}
J\left(H_{\xi^*}\oplus V_{C^*}\right)=V_{C^*}\oplus H_{\xi^*}\,,\,J\left(\{\xi^*\}\right)=\{C^*\}=\left(\{C^*\}^{\perp}\right)^{\perp}
\end{displaymath}
which end's the proof.
\end{proof}

We recall that a $r$-dimensional distribution $D$ on a Riemannian manifold $(M, g)$ is
minimal if the mean-curvature vector field $H$ of $D$ vanishes indentically, where
\[H =\frac{1}{r}\sum_{i=1}^r(\nabla_{X_i}X_i)^\perp
,\]where $\nabla$ is the Levi-Civita connection on $(M, g)$, $\{X_1,\ldots,X_r\}$, is an orthonormal frame
of $D$, and $(\nabla_XY)^\perp$ denotes the component of $\nabla_XY $ in the orthogonal complementary
distribution $D^\perp$ of $D$ in $TM$.

It is well known, see \cite{C1}, that the totally real distribution of a $CR$--submanifold of an (almost) K\"{a}hler manifold is integrable and its maximal complex (holomorphic) distribution is minimal. Then we obviously have that
the line distribution $\{\xi^*\}$ is integrable, and the distribution $H_{\xi^*}\oplus V_{C^*}$ is minimal.

Let $\omega_i$  be the dual $1$--forms of the vertical vector fields $\frac{\overline{\partial}}{\overline{\partial}p_i}$ and $\theta_i$ be the dual $1$--forms of the horizontal vector fields $\frac{\overline{\delta}}{\overline{\delta}x_i}$, that is $\omega_i\left(\frac{\overline{\partial}}{\overline{\partial}p_j}\right)=\delta^j_i$ and $\theta_i\left(\frac{\overline{\delta}}{\overline{\delta}x_j}\right)=\delta^j_i$, respectively. It is easy to see that we have the following relations:
\begin{equation}
\label{r5}
\delta p_i=P^j_i\omega_j\,\,\,{\rm and}\,\,\,dx^i=h^{ij}\theta_j.
\end{equation}

As we already noticed the vertical vector fields $\left\{\frac{\overline{\partial}}{\overline{\partial}p_i}\right\}$, $i=1,\ldots,n$ are linear dependent and we consider the linear independent system $\left\{\frac{\overline{\partial}}{\overline{\partial}p_1},\ldots,\frac{\overline{\partial}}{\overline{\partial}p_{i-1}},\frac{\overline{\partial}}{\overline{\partial}p_{i+1}},\ldots,\frac{\overline{\partial}}{\overline{\partial}p_{n}}\right\}$ that generates $V_{C^*}$. Consequently, by means of $J$, we get the linear independent system of horizontal vector fields $\left\{\frac{\overline{\delta}}{\overline{\delta} x_1},\ldots,\frac{\overline{\delta}}{\overline{\delta} x_{i-1}},\frac{\overline{\delta}}{\overline{\delta} x_{i+1}},\ldots,\frac{\overline{\delta}}{\overline{\delta} x_n}\right\}$  that generates $H_{\xi^*}$

Then the general theory for cohomology of $CR$-submanifolds of an (almost) K\"{a}hler manifolds, \cite{C1}, leads to
\begin{theorem}
The differential form
\begin{displaymath}
\nu=\omega_1\wedge\ldots\wedge\widehat{\omega_{i}}\wedge\ldots\wedge\omega_n\wedge\theta_{1}\wedge\ldots\wedge\widehat{\theta_{i}}\wedge\ldots\wedge\theta_n
\end{displaymath}
is closed and it defines a cohomology class
\begin{equation}
[\nu]\in H^{2n-2}\left(I(M,K)(c)\right).
\label{III11}
\end{equation}
\end{theorem}
\begin{definition}
The cohomology class $[\nu]$ is called the canonical class of the $c$--indicatrix cotangent bundle $I(M,K)(c)$ of a Cartan space $(M,K)$.
\end{definition}
\begin{remark}
The form $\nu$ which defines the canonical class can be expressed in the form 
\begin{displaymath}
\nu=\frac{(-1)^{n-1}}{(n-1)!}\left(i^*\Omega\right)^{n-1},
\end{displaymath}
where $\Omega$ is the fundamental form given in \eqref{I12}.
\end{remark}
Since $I(M,K)(c)$ is compact when $M$ is compact, acording to \cite{C1} we have
\begin{corollary}
If the cohomology groups $H^{2k}(I(M,K)(c))=0$ for some $k<n$ then either holomorphic distribution $H_{\xi^*}\oplus V_{C^*}$ is not integrable or its totally real distribution $\{\xi^*\}$ is not minimal.
\end{corollary}

We notice that the Poincar\'{e}-Cartan $2$--form associated to a Finsler function has rank $2n-2$ and it play an important role in projective metrizability problem in Finsler geometry, see \cite{B-Mu, C-M-S}.  In the end of this section we prove that the differential form $\nu$ that represents the canonical class of the  $c$--indicatrix cotangent bundle $I(M,K)(c)$ of a Cartan space $(M,K)$ can be related in terms of the Poincar\'{e}-Cartan $2$--form associated to the Cartan function $K$ which is defined as in Finsler geometry. More exactly, using the Fr\"{o}licher-Nijenhuis formalism, for the given complex structure $J$ from \eqref{I11} we consider the differential $d_J=i_J\circ d-d\circ i_J$, where $J$ is considered as a vector valued $1$--differential form. Then, using $dK=\zeta$, by direct calculus we obtain
\begin{displaymath}
d_JK=i_JdK=i_J\zeta=-g_{ij}\zeta^idx^j=-\frac{p_j}{K}dx^j.
\end{displaymath}
 Now, using $d_hK=0$ we obtain the following Poincar\'{e}-Cartan $2$--form associated to the Cartan function $K$  
\begin{equation}
dd_JK=-\frac{1}{K}P^i_j\delta p_i\wedge dx^j.
\label{r6}
\end{equation}
Using the relations \eqref{r5} and the fact that $P^i_jP^l_i=P^l_j$ and $P^l_jh^{jk}=h^{kl}$, we have the following expression of the Poincar\'{e}-Cartan $2$--form
\begin{equation}
\label{r7}
dd_JK=-\frac{1}{K}h^{lk}\omega_l\wedge\theta_k.
\end{equation}
Finally, taking the $n-1$ power of the above $2$-form, we get
\begin{displaymath}
\nu=\frac{(-1)^{\frac{n(n-1)}{2}}K^{n-1}}{(n-1)!\det \widetilde{H}}(dd_JK)^{n-1},
\end{displaymath}
where $\widetilde{H}=(h^{kl})$, $l,k\in\{1,\ldots,i-1,i+1,\ldots,n\}$.

\section*{Acknowledgment}
The authors cordially thank to Referee(s) for several useful remarks and suggestions about the initial Submissions which improve substantially the presentation and the contents of this paper.

\noindent
Cristian Ida and Adelina Manea\\
Department of Mathematics and Computer Science\\
University Transilvania of Bra\c{s}ov\\
Address: Bra\c{s}ov 500091, Str. Iuliu Maniu 50, Rom\^{a}nia\\
email: \textit{cristian.ida@unitbv.ro; amanea28@yahoo.com}\\
\\

\smallskip

\end{document}